\documentclass{amsart}
\usepackage{array}
\newcolumntype{P}[1]{>{\raggedright\let\newline\\\arraybackslash\hspace{0pt}}m{#1}}
\usepackage{graphicx,verbatim, amsmath, amssymb, amsthm, amsfonts, epsfig, amsxtra,ifthen,mathtools,epstopdf,caption,enumitem,hhline,bbm,capt-of,longtable}	
\usepackage[bookmarks=true, bookmarksopen=false,%
    colorlinks=true,%
    linkcolor=darkblue,%
    citecolor=darkblue,%
    filecolor=darkblue,%
    menucolor=darkblue,%
    linktoc=all,%
    urlcolor=darkblue
]{hyperref}
\usepackage{xcolor}
\definecolor{darkblue}{cmyk}{1,0.3,0,0.1}  
\DeclareFontFamily{U}{mathx}{\hyphenchar\font45}
\DeclareFontShape{U}{mathx}{m}{n}{<-> mathx10}{}
\DeclareSymbolFont{mathx}{U}{mathx}{m}{n}
\DeclareMathAccent{\widebar}{0}{mathx}{"73}
\DeclareGraphicsExtensions{.ps}
\newtheorem{proposition}{Proposition}[section]
\newtheorem{theorem}[proposition]{Theorem}

\newtheorem{lemma}[proposition]{Lemma}

\theoremstyle{definition}

\theoremstyle{remark}
\newtheorem{remark}[proposition]{Remark}

\numberwithin{equation}{section}

\newcommand{\newword}[1]{\textbf{\emph{#1}}}

\newcommand{\integers}{\mathbb Z}
\newcommand{\rationals}{\mathbb Q}

\newcommand{\reals}{\mathbb R}

\newcommand{\set}[1]{{\lbrace #1 \rbrace}}

\newcommand{\brr}[1]{{\bigl\langle #1 \bigr\rangle}}
\newcommand{\brrr}[1]{{\Bigl\langle #1 \Bigr\rangle}}

\newcommand{\F}{{\mathcal F}}

\newcommand{\ck}{\spcheck}

\newcommand{\Comp}{\mathrm{Comp}_C}

\newcommand{\g}{\mathbf{g}}
\newcommand{\gp}{\overline{\g}}

\renewcommand{\b}{\mathbf{b}}

\newcommand{\kk}{{\boldsymbol{k}}}
\renewcommand{\ll}{{\boldsymbol\ell}}

\renewcommand{\a}{\mathbf{a}}

\renewcommand{\v}{\mathbf{v}}

\newcommand{\tB}{{\widetilde{B}}}

\newcommand{\Fan}{\operatorname{Fan}}

\newcommand{\re}{\mathrm{re}}

\renewcommand{\d}{{\mathfrak d}}

\newcommand{\RSChar}{\Phi}
\newcommand{\RS}{\RSChar}

\newcommand{\AP}[1]{\RS_{#1}}
\newcommand{\APre}[1]{\AP{#1}^\re}

\newcommand{\margincolor}{red}      
\definecolor{darkgreen}{rgb}{0,0.7,0}

\newcounter{margincounter}
\newcommand{\marginnum}{
\ifnum\value{margincounter}<10
\textcolor{\margincolor}{\begin{picture}(0,0)\put(2.2,2.4){\circle{9}}\end{picture}\footnotesize\arabic{margincounter}}
\else\ifnum\value{margincounter}<100
\textcolor{\margincolor}{\begin{picture}(0,0)\put(4.256,2.5){\circle{11}}\end{picture}\footnotesize\arabic{margincounter}}
\else
\textcolor{\margincolor}{\begin{picture}(0,0)\put(6.8,2.5){\circle{14}}\end{picture}\footnotesize\arabic{margincounter}}
\fi\fi
}

\makeatletter
\newcommand{\switchmargin}{
\if@reversemargin
\normalmarginpar
\else
\reversemarginpar
\fi
}
\makeatother

\author{Nathan Reading}
\author{Salvatore Stella}
\title{Universal cluster algebras of affine type}
\title[Neighboring seeds, universal coefficients, finite mutation-type]{Neighboring seeds in affine type: \\Universal coefficients and finite mutation-type}
\address[Nathan Reading]{Department of Mathematics, NC State University, Raleigh, USA}
\address[Salvatore Stella]{Dipartimento di Ingegneria e Scienze dell'Informazione e Matematica, Universit\`{a} degli Studi dell'Aquila, IT}
\thanks{Nathan Reading was partially supported by the Simons Foundation under award number 581608 and by the National Science Foundation under award number DMS-2054489.
Salvatore Stella was partially supported by PRIN 20223FEA2E, \lq\lq Cluster algebras and Poisson Lie groups\rq\rq and INdAM/GNSAGA}

\allowdisplaybreaks

\begin{document}

\begin{abstract}
Neighboring seeds in a cluster algebra of affine type are seeds that are as close as possible to the boundary of the $\g$-vector fan.
This short note highlights the characterization of neighboring seeds given in a recent paper of Reading, Rupel, and Stella and applies that characterization in two ways.
We prove a conjecture on the construction of universal geometric cluster algebras of affine type.
We also characterize extended exchange matrices of affine type that are mutation-finite.
\end{abstract}

\maketitle


\setcounter{tocdepth}{1}
\tableofcontents

\section{Introduction}
Universal cluster algebras of finite type were constructed in \cite[Section~12]{ca4}.
They are the universal objects in the sense of coefficient specialization, among cluster algebras with a fixed exchange matrix of finite type.
The notion of universal cluster algebras (of geometric type) was extended to arbitrary exchange matrices in~\cite{universal}. 
In this paper, we construct universal geometric cluster algebras of affine type, confirming and adding detail to an earlier conjecture \cite[Conjecture~10.15]{universal}.
This expands the list of exchange matrices known to admit universal geometric coefficients to include rank~$2$ and finite type~\cite{universal}, certain marked surfaces~\cite{unisurface,unitorus,unisphere,unidread}, certain marked orbifolds~\cite{VielThesis}, and now affine type.
The previously known cases include the classical affine types, but here we give a uniform approach that includes all affine types (including twisted affine types).

The sticking point of the conjecture was a lack of detailed information about the space outside of the $\g$-vector fan in affine type (whose closure is called the \newword{imaginary wall} because it is a wall in the cluster scattering diagram~\cite{affscat}), and particularly the behavior of mutation maps in the imaginary wall.
That detailed information has now been provided in \cite{affdomreg} in the process of determining, in affine type, the \newword{dominance regions} \cite{RSDom} that mediate basis changes between the theta basis~\cite{GHKK} and other pointed bases~\cite{FanQin}.
Thus, this paper also serves to highlight a powerful proof technique that was central to \cite{affdomreg} and is central to this paper:  
Reducing facts about a cluster algebra of affine type (or a related algebraic or discrete-geometric object) to facts in the imaginary wall, where they essentially coincide with corresponding facts in a cluster algebra of finite type C.
The appearance of finite type C in the ``Coxeter-Catalan combinatorics'' of affine type goes back at least to McCammond and Sulway~\cite{McSul}, who completed the affine noncrossing partition poset to a lattice by adjoining a product of finite-type-B noncrossing partition lattices.
(In their setting and some other settings, the distinction between types B and C is meaningless.)
More recently, type-C-Catalan combinatorics appears inside Coxeter-Catalan combinatorics of affine type in \cite{affncA,affncD,affdomreg,affdenom,afftheta}.
Arguably, the appearance of finite type~C in analogous situations goes much further back in the representation theory of finite-dimensional algebras to the ``tubes'' that appear in AR-quivers of affine type, because the combinatorics near the bottoms of tubes recovers the simplicial cyclohedron (also known as the type-B associahedron) \cite{BottTaubes,Simion}.

Already in~\cite{affdenom}, the imaginary wall was given a fan structure that exhibits the combinatorics of the cyclohedron. 
The contribution of \cite{affdomreg} concerns \newword{neighboring seeds}, meaning seeds that are as close as possible to the imaginary wall.
The first key fact is that, when the initial seed is a neighboring seed, the imaginary wall not only has type-C combinatorics, but is \emph{metrically} the direct sum of $\g$-vector fans of finite type-C and a fan with only two cones: the zero cone and a ray.
The second key fact is that, for any mutation of the finite type-C fans, there is a corresponding sequence of mutations of the affine fan that accomplishes that type\nobreakdash-C mutation on the imaginary wall, while preserving the property that the initial seed is neighboring.
These two facts together allow us to use the known existence of universal geometric coefficients of finite type to complete the construction of universal geometric coefficients of affine type.
(Seeds that satisfy the description of neighboring seeds from \cite[Theorem~4.38]{affdomreg}---quoted here as Theorem~\ref{neigh B}---also appear in work of Felikson and Tumarkin~\cite{FeTuCoeff} and of Greenberg and Kaufman~\cite{GreenbergKaufman}.
See \cite[Remark~4.39]{affdomreg}.)

In this paper, we also take advantage of the detailed information about neighboring seeds provided by~\cite{affdomreg} to give a characterization of which extended exchange matrices of affine type are of finite mutation-type.
Extended exchange matrices of finite mutation-type were classified by Felikson and Tumarkin in~\cite{FeTuCoeff}.
They observe that a necessary condition for an extended exchange matrix $\tB$ to be of finite mutation type is that the underlying exchange matrix $B$ is of finite mutation type, and they determine which coefficient rows can be adjoined to $B$ while preserving mutation-finiteness.
Our result on extended exchange matrices of affine type and finite mutation-type is not really an application of results on neighboring seeds.
Rather, our result comes about because the Felikson--Tumarkin result uses neighboring seeds and because the information about neighboring seeds in~\cite{affdomreg} suggests how one should phrase the characterization in affine type:
If $B$ is of affine type, then $\tB$ is of finite mutation-type if and only if each coefficient row of $\tB$ is in the imaginary wall (more specifically, the closure of the boundary of the $\g$-vector fan for $B^T$). 
We prove the characterization using established tools for mutation maps of affine type, in the process making a minor correction to the affine case of the Felikson--Tumarkin result about extended exchange matrices of finite mutation-type from marked surfaces.
(See Remark~\ref{correction}.)

We now review the necessary background and the results on neighboring seeds and then state and prove the result on universal geometric coefficients of affine type and the mutation-finiteness result.

\section{Acyclic seeds in a cluster algebra of affine type}\label{acyc sec}
Let $B$ be an $n\times n$ \newword{exchange matrix} (a skew-symmetrizable integer matrix).
An exchange matrix $B=[b_{ij}]$ is \newword{acyclic} if, possibly after reindexing, it has the property that $b_{ij}\leq 0$ whenever $i>j$.
Whenever $B$ is acyclic, we will assume such an indexing.
An exchange matrix determines a Cartan matrix $A$ by changing the signs of all entries to nonpositive and then replacing each $0$ on the diagonal with $2$.
An exchange matrix $B$ is of finite type if it is mutation equivalent (i.e. related by a sequence of mutations) to an exchange matrix whose associated Cartan matrix is of finite type.
Equivalently, the cluster algebra determined by $B$ has finitely many seeds.
An exchange matrix is of \newword{affine type} if it is mutation-equivalent to an \emph{acyclic} exchange matrix whose associated Cartan matrix is of affine type.
When $B$ is $3\times 3$ or larger, it is of affine type if and only if it is not of finite type but the number of seeds within distance $k$ of the initial seed is bounded above by a linear function of $k$.
(Combine \cite[Theorem~3.5]{Seven} and \cite[Theorem~1.1]{FeShThTu12}.)

We assume the usual notion of matrix mutation.
If ${\kk=k_rk_{r-1}\cdots k_1}$ is a sequence of indices, then $\mu_\kk(B)$ means $\mu_{k_r}(\mu_{k_{r-1}}(\cdots(\mu_{k_1}(B))\cdots))$.
The \newword{mutation map} ${\eta^B_\kk:\reals^n\to\reals^n}$ takes the input vector in $\reals^n$, places it as an additional row below $B$, mutates the resulting matrix according to the sequence $\kk$, and outputs the bottom row of the mutated matrix.
There is a fan $\F_B$ called the \newword{mutation fan} for $B$ \cite[Definition~5.12]{universal}.
Roughly speaking, the fan $\F_B$ is determined by the common domains of linearity of all mutation maps.
For any sequence $\kk$, the mutation map $\eta^B_\kk$ is a piecewise linear isomorphism from $\F_B$ to $\F_{\mu_\kk(B)}$, linear on every cone of~$\F_B$.

Mutation fans of acyclic affine type can be understood through two combinatorial tools: the doubled Cambrian fans of~\cite{afframe} and the affine almost positive roots model of~\cite{affdenom}.
We will use results of \cite{afframe,affdenom,affscat,affdomreg} where the roles of $B$ and its transpose $B^T$ are reversed relative to what we need here.
We will quote the results as they appear in those papers and then apply a transpose when we use them.

Suppose $B$ is an acyclic exchange matrix.
Let $V$ be a real vector space with a basis $\alpha_1,\ldots,\alpha_n$, which we take to be the simple roots of the Kac-Moody root system~$\RS$ associated to $A$.
The \newword{root lattice} is the lattice in $V$ spanned by $\alpha_1,\ldots,\alpha_n$.
There is also a basis of simple co-roots $\alpha\ck_1,\ldots,\alpha\ck_n$, obtained from the simple roots by a certain scaling.
The dual vector space $V^*$ has a basis $\rho_1,\ldots,\rho_n$ of fundamental weights, and this basis is dual to the basis of simple co-roots in~$V$. 
The \newword{weight lattice} is the lattice in $V^*$ spanned by $\rho_1,\ldots,\rho_n$.
We identify $V^*$ with $\reals^n$ by identifying vectors in $\reals^n$ with fundamental-weight coordinates in $V^*$.

The simple roots determine simple reflections $S=\set{s_1,\ldots,s_n}$ that generate the Coxeter group $W$ associated to $A$.
The sign information in $B$ that is forgotten in order to determine $A$ amounts to the choice of a Coxeter element $c$ of $W$, and thus the information in $B$ amounts to the information in $A$ and $c$.

Now suppose in addition that $B$ is of affine type.
As usual, we write~$\delta$ for the unique positive imaginary root in $\RS$ that is not a positive multiple of other roots.
One can define a $c$-Cambrian fan in $V^*$, in the general sense of \cite[Section~9]{typefree}.
The $c$-Cambrian fan covers the Tits cone determined by $W$ and is contained in the $\g$-vector fan associated to $B$.
The \newword{doubled Cambrian fan} is the union of the $c$-Cambrian fan and the antipodal image of the $(c^{-1})$-Cambrian fan.
Since $B$ is of affine type, the Tits cone is a halfspace in $V^*$ bounded by $\delta^\perp$ and the doubled Cambrian fan coincides with the $\g$-vector fan of~$B$ \cite[Corollary~1.3]{afframe}.
As a result, for $B$ of acyclic affine type, the $g$-vector fan of $B$ covers all of $V^*$ except for the relative interior of a codimension-$1$ cone $\d_\infty$ that we call the \newword{imaginary wall}, because it is a wall in the cluster scattering diagram~\cite{affscat} in the sense of~\cite{GHKK}.

The affine almost positive roots model also depends on $B$ via $A$ and $c$.
It begins with a subset $\AP{c}$ of the root system $\RS$ determined by $A$ called the \newword{almost positive Schur roots}.
The negative roots in $\AP{c}$ are the negatives $-\alpha_1,\ldots,-\alpha_n$ of the simple roots.
The unique imaginary root in $\AP{c}$ is $\delta$, and we write $\APre{c}$ for $\AP{c}\setminus\set\delta$.
There is a compatibility relation on $\AP{c}$ such that the nonnegative spans of pairwise compatible subsets of $\AP{c}$ are the cones in a complete simplicial fan $\Fan_c(\RS)$ in~$V$.
The maximal pairwise compatible subsets are called \newword{$c$-clusters}.
More specifically, \newword{imaginary $c$-clusters}, those that contain~$\delta$, have size $n-1$ and \newword{real $c$-clusters}, those that don't contain~$\delta$, have size $n$.
The vectors in any real $c$-cluster are a $\integers$-basis for the root lattice \cite[Proposition~5.14(2)]{affdenom} and the vectors in any imaginary $c$-cluster are a $\integers$-basis for the intersection of the root lattice in $V$ with the $\reals$-linear span of the imaginary $c$-cluster \cite[Proposition~5.14(6)]{affdenom}.

There is a piecewise linear map $\nu_c:V\to V^*$ that is linear on every cone of $\Fan_c(\RS)$ and thus induces a fan structure $\nu_c(\Fan(\RS))$ on $V^*$,
The fan $\nu_c(\Fan(\RS))$ coincides with the mutation fan~$\F_{B^T}$ \cite[Theorem~2.9]{affscat}.
The map $\nu_c$ is a bijection from the root lattice in $V$ to the weight lattice in~$V^*$.

The (simple-root coordinates of) roots $\AP{c}$ are the denominator vectors of cluster variables associated to $B$ \cite[Theorem~1.2]{affdenom} and of the theta function indexed by $\nu_c(\delta)$ \cite[Corollary~3.6]{afftheta}.
The (fundamental-weight coordinates of) the vectors $\nu_c(\APre{c})$ are the $\g$-vectors of cluster variables associated to $B$ \cite[Proposition~9]{Rupel}.
Thus the unique ray of $\F_{B^T}$ that is not in the $\g$-vector fan is spanned by $\nu_c(\delta)$, and the star of that ray (the set of cones containing that ray) is a fan such that the union of its cones is the imaginary wall.
The cones of in this star are called \newword{imaginary cones}.
Maximal imaginary cones are of dimension $n-1$.

We will need the following lemma.

\begin{lemma}\label{int span}
Suppose $B$ is acyclic of affine type.
Suppose $C$ is a cone in $\F_{B^T}$ and $\v_1,\ldots,\v_k$ are the shortest integer vectors in the rays of $C$.
If $\v$ is an integer point in $C$, then $\v$ is a nonnegative $\integers$-linear combination of $\v_1,\ldots,\v_k$.
\end{lemma}
\begin{proof}
We may as well take $C$ to be a maximal cone.
If $C$ is a cone in the $\g$-vector fan, then $C$ is the image, under $\nu_c$, of the cone spanned by some real $c$-cluster.
The vectors in the real $c$-cluster are a $\integers$-basis for the root lattice in $V$.
Otherwise, $C$ is a cone in the star of the imaginary ray, and is the image under $\nu_c$ of some imaginary $c$-cluster.
The vectors in the imaginary $c$-cluster are a $\integers$-basis for the intersection of the root lattice in $V$ with the $\reals$-linear span of the imaginary $c$-cluster.
Since $\nu_c$ takes the root lattice in $V$ bijectively to the weight lattice in $V^*$, we see in either case that~$\v$ is a nonnegative $\integers$-linear combination of $\v_1,\ldots,\v_k$.
\end{proof}

\section{Neighboring seeds in a cluster algebra of affine type}\label{neigh sec}
Now let $B$ be an arbitrary exchange matrix of affine type, not necessarily acyclic.
Since each mutation map is an isomorphism of mutation fans, we know from the results quoted in Section~\ref{acyc sec} that $\F_{B^T}$ has a unique ray, the \newword{imaginary ray}, that is not contained in any full-dimensional cone, and that the star of that ray consists of $(n-1)$-dimensional \newword{imaginary cones} and their faces.
There is a vector~${\delta^B\in V}$ such that $-\frac12B\delta^B$ is the shortest integer vector in the imaginary ray and such that all imaginary cones are contained in~$(\delta^B)^\perp$ \cite[Proposition~4.24]{affdomreg}.
Thus we call the union of the imaginary cones the \newword{imaginary wall}.
When $B$ is acyclic, the vector~$\delta^B$ is the same as the vector $\delta$ from Section~\ref{acyc sec}.

Each imaginary cone is simplicial and has $n-2$ rays spanned by $\g$-vectors of cluster variables, in addition to the imaginary ray.
A \newword{neighboring seed} is a seed such that $n-2$ of its $\g$-vectors are in the imaginary wall (necessarily all in the same imaginary cone).
The property of being neighboring depends only on the exchange matrix \cite[Lemma~4.36]{affdomreg}, so we will refer to \newword{neighboring exchange matrices}.

A \newword{quasi-leaf} is a column $i$ of an exchange matrix $B$ that has at most two nonzero entries and such that, if $b_{ji}\neq0$ and $b_{ki}\neq0$, then the restriction of $B$ to rows and columns $i$, $j$, and $k$ is $\pm\begin{bsmallmatrix*}[r]0&1&-1\\-1&0&1\\1&-1&0\end{bsmallmatrix*}$.
As we describe block decompositions of matrices, an \newword{empty block} is a block with $0$ columns and/or $0$ rows.
The following is \cite[Theorem~4.38]{affdomreg}.

\begin{theorem}\label{neigh B}
Suppose $B$ is an exchange matrix of affine type.
Then the following conditions are equivalent.
\begin{enumerate}[label=\rm(\roman*), ref=(\roman*)]
\item \label{neigh}
$B$ is a neighboring exchange matrix.
\item \label{aff 2}
There exist indices $i$ and $j$ such that $\begin{bsmallmatrix}0&b_{ij}\\b_{ji}&0\end{bsmallmatrix}$ is of affine type.
\item \label{neigh detailed}
Up to relabeling, $B$ is
    $
      \begin{bsmallmatrix}
        B_{11} 	& 0 		& 0 		& B_{14} \\ 
        0 		& B_{22} 	& 0 		& B_{24} \\
        0 		& 0 		& B_{33} 	& B_{34} \\
        B_{41} 	& B_{42} 	& B_{43} 	& B_{44}
      \end{bsmallmatrix},
    $ 
where $B_{44}$ is a rank-$2$ exchange matrix of affine type, the square matrices $B_{11}$, $B_{22}$, and $B_{33}$ are in order of increasing size, and the following conditions hold for ${\ell\in\set{1,2,3}}$.
\begin{itemize}
\item
$B_{\ell\ell}$ is either an empty block or an exchange matrix of finite type A. 
\item
If $B_{\ell\ell}$ is nonempty, then its last column is a quasi-leaf of $B_{\ell\ell}$.
\item
If $B_{\ell4}$ is nonempty, then it is nonzero only in its last row.
\item
If $B_{4\ell}$ is nonempty, then it is nonzero only in its last column.
\item
If $B_{\ell4}$ is nonempty (equivalently if $B_{4\ell}$ is nonempty), then the nonzero rows and columns in $\begin{bsmallmatrix} 0 & B_{\ell4} \\ B_{4\ell} & B_{44} \end{bsmallmatrix}$ are of the form
$\begin{bsmallmatrix*}[r]
	0 & 1 & -1 \\
	-1 & 0 & 2 \\
	1 & -2 & 0
	\end{bsmallmatrix*}$,
$\begin{bsmallmatrix*}[r]
	0 & 2 & -2 \\
	-1 & 0 & 2 \\
	1 & -2 & 0
	\end{bsmallmatrix*}$,
$\begin{bsmallmatrix*}[r]
	0 & 1 & -1 \\
	-2 & 0 & 2 \\
	2 & -2 & 0
	\end{bsmallmatrix*}$,
$\begin{bsmallmatrix*}[r]
	0 & 3 & -3 \\
	-1 & 0 & 2 \\
	1 & -2 & 0
	\end{bsmallmatrix*}$,
$\begin{bsmallmatrix*}[r]
	0 & 1 & -1 \\
	-3 & 0 & 2 \\
	3 & -2 & 0
	\end{bsmallmatrix*}$,
$\begin{bsmallmatrix*}[r]
	0 & 1 & -2 \\
	-2 & 0 & 4 \\
	1 & -1 & 0
	\end{bsmallmatrix*}$, or
$\begin{bsmallmatrix*}[r]
	0 & 2 & -1 \\
	-1 & 0 & 1 \\
	2 & -4 & 0
	\end{bsmallmatrix*}$.\\
If the matrix contains an entry with absolute value greater than $2$, then $B_{11}$ and $B_{22}$ are empty.
\end{itemize}
\end{enumerate}
\end{theorem}

There is a \newword{special index} for each $\ell=1,2,3$ such that $B_{\ell\ell}$ is nonempty:  It is the index of the last column/row of $B_{\ell\ell}$.
The two indices of $B_{44}$ are the \newword{affine indices}.
 Given a neighboring exchange matrix $B$, the \newword{type-C companion} of $B$ is the $(n-2)\times(n-2)$ exchange matrix $\Comp(B)$ that agrees with the non-affine rows and columns of $B$, but with all entries in special columns multiplied by~$2$.
The matrix $\Comp(B)$ is block-diagonal, and \cite[Proposition~4.50]{affdomreg} says that each diagonal block is of finite type C.
To avoid cumbersome notation, we will identify vectors in $\reals^n$ that have affine entries zero with vectors in $\reals^{n-2}$, and specifically, we will apply mutation maps $\eta^{\Comp(B)^T}$ to such vectors.
The following propositions is the combination of parts of \mbox{\cite[Proposition~4.52]{affdomreg}} and \cite[Proposition~4.53]{affdomreg}.
(The two propositions in \cite{affdomreg} distinguish between the cases where $k$ is special or not and give more detail.)

\begin{proposition}\label{nice mut}
Suppose $B$ is a neighboring exchange matrix and $k$ is a non-affine index of~$B$.
Then there exists a sequence $\ll$ of indices in $\set{k,n-1,n}$ with the following properties.
\begin{enumerate}[label=\bf\arabic*., ref=\arabic*]
\item
$\mu_\ll(B)$ is neighboring.
\item
$\Comp(\mu_\ll(B))=\mu_k(\Comp(B))$.
\item
The mutation map $\eta^{B^T}_\ll$ fixes the imaginary ray pointwise.
\item
If $x$ is a vector whose affine entries are zero, then the affine entries of $\eta^B_\ll(x)$ are again zero and $\eta^{B^T}_\ll(x)=\eta^{\Comp(B)^T}_k(x)$.
\end{enumerate}
\end{proposition}

Proposition~\ref{nice mut} suggests that, for each sequence $\kk$ of non-affine indices, we define an \newword{expanded sequence}~$\widehat\kk$ by replacing each index in $\kk$ with the sequence that exists in light of the proposition.
Thus $\Comp(\mu_{\widehat\kk}(B))=\mu_\kk(\Comp(B))$, the mutation map $\eta^{B^T}_{\widehat\kk}$ fixes the imaginary ray pointwise, and if $x$ has affine entries zero, then $\eta^{B^T}_{\widehat\kk}(x)$ also has affine entries zero and $\eta^{B^T}_{\widehat\kk}(x)=\eta^{\Comp(B)^T}_\kk(x)$.

The following proposition is \cite[Proposiion~4.54]{affdomreg}.

\begin{proposition}\label{neigh im wall}
Suppose $B$ is neighboring and is indexed as in Theorem~\ref{neigh B}.
Then $\d^B_\infty$ is the half-hyperplane contained in $(\delta^B)^\perp$, containing the vector $-\frac12B\delta^B$, with relative boundary the codimension-$2$ space consisting of vectors that are zero in the affine indices.
\end{proposition}

The proof of Proposition~\ref{neigh im wall} in \cite{affdomreg} establishes a more detailed fact, given as the following proposition.  
We continue to identify vectors with affine indices zero with vectors in $\reals^{n-2}$, so that $\F_{\Comp(B)^T}$ can be considered as a set of cones in $\reals^n$.

\begin{proposition}\label{neigh im cone}
Suppose $B$ is neighboring.
Then $\F_{B^T}$ restricts to a fan in $\d^B_\infty$, consisting of all cones of $\F_{\Comp(B)^T}$, in the boundary of $\d^B_\infty$, and all cones obtained as the Minkowski sum of a cone of $\F_{\Comp(B)^T}$ plus the imaginary ray.
\end{proposition}

If $B$ is neighboring and  $\lambda\in\d_\infty^B$, Proposition~\ref{neigh im wall} says that we can uniquely write $\lambda=\lambda_0+\lambda_\infty$ such~$\lambda_0$ has affine entries zero and $\lambda_\infty$ is in the imaginary ray.
The following proposition is \cite[Proposiion~4.55]{affdomreg}.

\begin{proposition}\label{factor eta}
Suppose $B$ is neighboring and $\lambda\in\d_\infty^B$ with $\lambda=\lambda_0+\lambda_\infty$ as above.
If $\kk$ is a sequence of non-affine indices with expanded sequence $\widehat\kk$, then ${\eta^{B^T}_{\widehat\kk}(\lambda)=\eta_\kk^{\Comp(B)^T}(\lambda_0)+\lambda_\infty}$.
\end{proposition}

We conclude this section by quoting an additional result \cite[Proposition~4.49]{affdomreg} that, when combined with Proposition~\ref{neigh im wall}, gives a precise description of the imaginary ray and imaginary wall.

\pagebreak[4]

\begin{proposition}\label{neigh good stuff}
Suppose $B$ is a neighboring exchange matrix.
\begin{enumerate}[label=\bf\arabic*., ref=\arabic*]
\item\label{neigh delta}
The vector $\delta^B$ is zero in all non-affine indices.
Its affine entries are
\begin{itemize}
\item
$1,1$ if the affine submatrix of $B$ is $\begin{bsmallmatrix*}[r]0&2\\-2&0\end{bsmallmatrix*}$,
\item
$2,1$ if the affine submatrix of $B$ is $\begin{bsmallmatrix*}[r]0&4\\-1&0\end{bsmallmatrix*}$, or
\item
$1,2$ if the affine submatrix of $B$ is $\begin{bsmallmatrix*}[r]0&1\\-4&0\end{bsmallmatrix*}$.
\end{itemize}
\item\label{neigh im ray}
The vector $-\frac12B\delta^B$ (the shortest integer vector that spans the imaginary ray) is zero in all non-affine entries.
Its affine entries are
\begin{itemize}
\item
$-1,1$ if the affine submatrix of $B$ is $\begin{bsmallmatrix*}[r]0&2\\-2&0\end{bsmallmatrix*}$,
\item
$-2,1$ if the affine submatrix of $B$ is $\begin{bsmallmatrix*}[r]0&4\\-1&0\end{bsmallmatrix*}$, or
\item
$-1,2$ if the affine submatrix of $B$ is $\begin{bsmallmatrix*}[r]0&1\\-4&0\end{bsmallmatrix*}$.
\end{itemize}
\end{enumerate}
\end{proposition}

\section{Universal geometric coefficients in affine type}\label{univ sec}
We now describe an application of the results of Sections~\ref{acyc sec} and~\ref{neigh sec} to mutation-linear algebra, or more specifically to the study of universal geometric coefficients for cluster algebras.
Mutation-linear algebra originated in \cite{universal} and has been studied in  \cite{unisurface,unitorus,unisphere,dominance}.
Here we follow the notation of \cite{dominance}.

If $B$ is an exchange matrix not having a row (equivalently column) of zeros, then a \newword{$B$-coherent linear relation} among vectors in $\reals^n$ is a linear relation that is still a linear relation after applying the same arbitrary mutation map to all the vectors.
In other words, it is an expression $\sum_{i\in S}c_i\v_i$ with property that $\sum_{i\in S}c_i\eta_\kk^B(\v_i)$ is the zero vector for every sequence~$\kk$.
(The full definition of a $B$-coherent linear relation \cite[Definition~4.2]{universal} has an additional condition that is redundant unless~$B$ has a row of zeros.)
The \newword{support} of a $B$-coherent linear relation $\sum_{i\in S}c_i\v_i$ is $\set{\v_i:c_i\neq0}$ and the $B$-coherent linear relation is \newword{trivial} if its support is empty.

Suppose $R$ is either the integers $\integers$ or some subfield of $\reals$ that contains $\rationals$ as a subfield.
We write $R^B$ for the \newword{mutation-linear structure} on $R^n$, which is the set $R^n$ and the collection of $B$-coherent linear relations.
One can write much of the usual linear algebra in terms of the set of all linear relations among vectors in $R^n$;
\newword{mutation-linear algebra} results from using the same constructions, but with the set of $B$-coherent linear relations replacing the larger set of all linear relations.

The main thrust of \cite{universal} is to find a \newword{basis for $R^B$} (called an \newword{$R$-basis for $B$} in~\cite{universal}).
This is a set of vectors that is \newword{independent}, meaning that any $B$-coherent linear relation among basis vectors is trivial, and \newword{spanning}, meaning that for any $\a\in R^n$, there is a finite set $\set{\b_i:i\in S}$ of basis vectors and coefficients $c_i\in R$ such that $\a-\sum_{i\in S}c_i\b_i$ is a $B$-coherent linear relation.
A \newword{positive basis for $R^B$} is a basis such that for every $\a\in R^n$, the unique $B$-coherent linear relation $\a-\sum_{i\in S}c_i\b_i$ expressing $\a$ in terms of basis vectors has $c_i\ge0$ for all $i\in S$.

Finding a basis for $R^B$ is equivalent \cite[Theorem~4.4]{universal} to finding a \newword{universal geometric cluster algebra} with exchange matrix $B$, meaning a cluster algebra of geometric type with exchange matrix $B$ such that any other cluster algebra of geometric type with exchange matrix $B$ is obtained by a unique coefficient specialization.
For this to make sense, we work with a notion of cluster algebras of geometric type \cite[Section~2]{universal} that generalizes the usual notion \cite[Section~2]{ca4} by allowing infinitely many coefficient rows, which may have non-integer entries (when~$R$ is strictly larger than~$\integers$).
Finding a basis for $R^B$ is less formally referred to as finding \newword{universal geometric coefficients} for $B$.

When $B$ is of finite type, a basis for $R^B$ consists of the $\g$-vectors of cluster variables associated to~$B^T$.
(This is \cite[Theorem~10.12]{universal}, which is a version of \cite[Theorem~12.4]{ca4}.)
It was conjectured \cite[Conjecture~10.15]{universal} that for any exchange matrix $B$ of affine type, there is a unique vector $\v_\infty$ such that, for any choice of $R$, the $\g$-vectors of cluster variables for $B^T$ and the vector $\v_\infty$ constitute a positive basis for $R^B$.
We will prove the following theorem, which proves \cite[Conjecture~10.15]{universal} and identifies $\v_\infty$ explicitly.

\begin{theorem}\label{uniaffine}
Suppose $B$ is an exchange matrix of affine type.
Then the set of $\g$-vectors of cluster variables associated to $B^T$, together with the vector $-\frac12B^T\delta^{B^T}$, constitute a positive basis for $R^B$ for every $R$.
\end{theorem}

Some lemmas will be useful in the proof of Theorem~\ref{uniaffine}.
The first two are immediate from the definition (and closely related) and the third is \cite[Proposition~5.9]{universal}.  

\begin{lemma}\label{I can't believe I've never written this down}
For any sequence $\kk$ of indices, an expression $\sum_{i\in S}c_i\v_i$ is a $B$-coherent linear relation if and only if $\sum_{i\in S}c_i\eta_\kk^B(\v_i)$ is a $\mu_\kk(B)$-coherent linear relation.
\end{lemma}

\begin{lemma}\label{wow, also never wrote this down}
For any sequence $\kk$ of indices, a collection $(\b_i:i\in I)$ of vectors in~$R^n$ is a (positive) basis for $R^B$ if and only if $(\eta_\kk^B(\b_i):i\in I)$ is a (positive) basis for $R^{\mu_\kk(B)}$.
\end{lemma}

\begin{lemma}\label{local coherent} 
If $S$ is finite, $\sum_{i\in S}c_i\v_i$ equals zero, and $\set{\v_i:i\in S}$ is contained in some cone of the mutation fan, then $\sum_{i\in S}c_i\v_i$ is a $B$-coherent linear relation.
\end{lemma}

The following proposition is the specialization of \cite[Proposition~10.1]{universal} to affine type.
(The statement in \cite{universal} is conditional on the ``Standard Hypotheses'', but these are now a general theorem, namely ``sign-coherence of $C$-vectors''.)

\begin{proposition}\label{Tits B-coherent}
Suppose $B$ is acyclic of affine type.
Then any $B$-coherent linear relation supported on the $\g$-vectors of cluster variables associated to $B^T$ and the vector $-\frac12B^T\delta^{B^T}$ is in fact supported on vectors in~$\bigl(\delta^{B^T}\bigr)^\perp$.
\end{proposition}

Recall that when $B$ is acyclic we assume that it is indexed with $b_{ij}\leq 0$ whenever $i>j$.
Suppose $x\in V^*\setminus\d_\infty^B$.
The third assertion of \cite[Proposition~4.17]{affdomreg} says that there exists an integer~$P$ such that $\brr{\bigl(\eta_{12\cdots n}^{B^T}\bigr)^p(x),\delta}>0$ whenever $p\ge P$.
The fourth assertion of the same proposition immediately implies that the $\bigl(\eta_{12\cdots n}^{B^T}\bigr)$-orbit of $x$ is infinite.
We will need a version of these facts where we exchange the roles of $B$ and~$B^T$.
Note that by reversing the roles of $B$ and $B^T$, we also reverse the indexing.
Thus, we have the following proposition.

\begin{proposition}\label{everything to Tits}
Suppose $B$ is acyclic of affine type and $x\in V^*\setminus\d_\infty^{B^T}$.
Then the $\bigl(\eta_{n(n-1)\cdots1}^B\bigr)$-orbit of $x$ is infinite.
Furthermore, there exists an integer~$P$ such that $\brrr{\bigl(\eta_{n(n-1)\cdots1}^B\bigr)^p(x),\delta^{B^T}}>0$ whenever $p\ge P$.
\end{proposition} 

\begin{proposition}\label{im wall B-coherent}
Suppose $B$ is acyclic of affine type.
Then any $B$-coherent linear relation supported on the $\g$-vectors of cluster variables associated to $B^T$ and the vector $-\frac12B^T\delta^{B^T}$ is in fact supported on vectors in~$\d_\infty^{B^T}$.
\end{proposition}
\begin{proof}
Consider a $B$-coherent linear relation $\sum_{i\in S}c_i\v_i$.
Since $B$ is acyclic and indexed as described above, $\mu_{n(n-1)\cdots1}(B)=B$.
Thus Lemma~\ref{I can't believe I've never written this down} implies that $\sum_{i\in S}c_i\bigl(\eta^B_{(n(n-1)\cdots1)}\bigr)^p(\v_i)$ is a $B$-coherent linear relation for any $p\ge0$.
By Proposition~\ref{everything to Tits}, since the sum is finite, there exists $P$ such that whenever $p\ge P$, we have ${\brrr{\bigl(\eta^B_{(n(n-1)\cdots1)}\bigr)^p(\v_i),\delta^{B^T}}>0}$ for all $i\in S$ such that $\v_i\not\in\d_\infty^{B^T}$.
Therefore, by Proposition~\ref{Tits B-coherent}, we conclude that $c_i=0$ for all $i\in S$ such that $\v_i\not\in\d_\infty^{B^T}$.
\end{proof}

We now prove the main result of the section.

\begin{proof}[Proof of Theorem~\ref{uniaffine}]
For convenience in this proof, we write $\g(B)$ for the set of $\g$-vectors of cluster variables associated to an exchange matrix $B$ and, when $B$ is of affine type, $\gp(B)$ for $\g(B)\cup\set{-\frac12B\delta^B}$.

Suppose $B$ is of affine type.
By Lemma~\ref{wow, also never wrote this down}, we can assume, for now, that $B$ is acyclic.
We first show that $\gp(B)$ is a spanning set, in the sense of mutation-linear algebra.
Suppose $\a\in R^n$.
Then $\a$ is in some cone $C$ of the mutation fan $\F_B$.
If $R=\integers$, then Lemma~\ref{int span} says that $\a$ is a nonnegative $\integers$-linear combination of the shortest integer vectors in the rays of $C$.
Otherwise,~$R$ is a field, so $\a$ is again a nonnegative $R$-linear combination of these shortest vectors.
These shortest vectors are in $\gp(B^T)$.
Now Lemma~\ref{local coherent} implies that the linear relation that writes $\a$ as an $R$-linear combination of these vectors is $B$-coherent.

It remains to show that $\gp(B^T)$ is an independent set, in the sense of mutation-linear algebra.
In light of Proposition~\ref{im wall B-coherent}, we need to prove the following claim:
Any $B$-coherent linear relation supported on $\gp(B^T)\cap\d_\infty^{B^T}$ is trivial.

Up to now, we have been assuming that $B$ is acyclic.
Recall that each mutation map $\eta^B_\kk$ is a piecewise linear isomorphism from $\F_B$ to $\F_{\mu_\kk(B)}$, linear on every cone of~$\F_B$, so that $\eta_\kk^B$ maps the imaginary wall in $\F_B$ to the imaginary wall in $\F_{\mu_\kk(B)}$.
Thus by Lemma~\ref{wow, also never wrote this down} again, it is enough to prove the claim for any exchange matrix mutation-equivalent to~$B$.
Therefore, we now change our assumption on $B$ and take~$B$ to be a neighboring exchange matrix.

It is clear from Theorem~\ref{neigh B} that $B$ is neighboring if and only if $B^T$ is neighboring.
Thus we are free to quote results from Section~\ref{neigh sec} with the roles of $B$ and $B^T$ reversed.
Proposition~\ref{neigh im wall} says that $\d^{B^T}_\infty$ is the half-hyperplane contained in $(\delta^{B^T})^\perp$, containing the vector $-\frac12B^T\delta^{B^T}$, with relative boundary the codimension-$2$ space consisting of vectors that are zero in the affine indices.
Suppose $\kk$ is a sequence of non-affine indices with expanded sequence $\widehat\kk$.
(A technical point:  We determine $\widehat\kk$ using Proposition~\ref{nice mut} \emph{as it applies to $B^T$}.)
Proposition~\ref{factor eta} implies that $\eta^B_{\widehat\kk}(\lambda)$ fixes $-\frac12B^T\delta^{B^T}$ and acts as $\eta_\kk^{\Comp(B^T)^T}$ on $\g$-vectors in the relative boundary of~$\d^{B^T}_\infty$.
As an immediate consequence of Proposition~\ref{neigh im cone}, $\g(B^T)\cap\d^{B^T}_\infty\subseteq\g(\Comp(B^T))$ (continuing to identify $\reals^{n-2}$ with the vectors in $\reals^n$ with affine entries zero).  

Any $B$-coherent linear relation that is supported on $\gp(B^T)\cap\d^{B^T}_\infty$ restricts to a $\Comp(B^T)^T$-coherent linear relation on $\g(\Comp(B^T))$.
Since $\Comp(B^T)$ is of finite type, $\g(\Comp(B^T))$ is a basis for $R^{\Comp(B^T)^T}$, so this restriction has all coefficients zero.
The only remaining potentially nonzero coefficient is the coefficient on $-\frac12B^T\delta^{B^T}$, but this coefficient is then also zero because the $B$-coherent linear relation is in particular an ordinary linear relation.
\end{proof}

\section{Mutation-finiteness with coefficients in affine type}
An $n\times n$ exchange matrix $B=[b_{ij}]$ is of \newword{finite mutation-type} if there are only finitely many different matrices that arise as $\mu_\kk(B)$, as $\kk$ varies over all sequences of indices in $\set{1,\ldots,n}$.
Similarly, an extension $\tB$ of $B$ is of finite mutation type if the set of matrices~$\mu_\kk(\tB)$ is finite, again as $\kk$ varies over all sequences in $\set{1,\ldots,n}$.
The classification of extended exchange matrices of finite mutation-type was carried out by Felikson and Tumarkin in~\cite{FeTuCoeff}, building on the earlier classification of exchange matrices of finite mutation-type.
In particular, exchange matrices of affine type are known to be mutation-finite.
We will rephrase their results in terms of exchange matrices (rather than quivers) and focus on the case in the classification where $B$ is of affine type.
They call a vector \newword{admissible} for $B$ if the extended exchange matrix consisting of $B$ with that vector adjoined as a coefficient row is mutation-finite and observe that $\tB$ is of finite mutation-type if and only if (1) $B$ is of finite mutation type and (2) every coefficient row of $\tB$ is admissible.

The characterization of admissible vectors in affine type assumes that there exist indices $i$ and $j$ such that $\begin{bsmallmatrix}0&b_{ij}\\b_{ji}&0\end{bsmallmatrix}$ is of affine type.
In light of Theorem~\ref{neigh B}, this is if and only if $B$ is a neighboring exchange matrix.
Furthermore, combining \mbox{\cite[Theorem~4.4]{FeTuCoeff}}, \cite[Theorem~9.4]{FeTuCoeff}, and \cite[Theorem~9.6]{FeTuCoeff} with Proposition~\ref{neigh good stuff}, we obtain the following statement:
If $B$ is a neighboring exchange matrix of affine type, then a vector is admissible for $B$ if and only if the vector is contained in~$\d_\infty^{B^T}$.
(Note the transpose in $\d_\infty^{B^T}$.)
As noted in the proof of Theorem~\ref{uniaffine}, each mutation map $\eta_\kk^B$ maps the imaginary wall in $\F_B$ to the imaginary wall in $\F_{\mu_\kk(B)}$.
Thus the Felikson-Tumarkin result implies that if $B$ is an arbitrary exchange matrix of affine type, then a vector is admissible for $B$ if and only if the vector is contained in $\d_\infty^{B^T}$.
We will prove that result directly, using the combinatorics of exchange matrices of acyclic affine type.

\begin{theorem}\label{mutfin coeff}
Suppose $B$ is an arbitrary exchange matrix of affine type and $\tB$ is an extension of~$B$.
Then $\tB$ is of finite mutation-type if and only if every coefficient row of $\tB$ is contained in $\d_\infty^{B^T}$.
\end{theorem}

\begin{proof}
As explained above, we need to show that a vector $\a\in V^*$ is admissible if and only if it is in $\d_\infty^B$.
Also as explained above, it is enough to prove the theorem for one exchange matrix in each mutation-equivalence class.
Therefore, we are free to choose $B$ to be acyclic of affine type.
Proposition~\ref{everything to Tits} implies that if $\a\not\in\d_\infty^B$, then $\a$ is not admissible.

On the other hand, suppose $\a\in\d_\infty^B$.
Then $\a$ is contained in some imaginary cone~$C$ of $\F_B$, and the cone $C$ is the nonnegative linear span of a linearly independent collection of vectors that span rays contained in $\d_\infty^B$.
Mutation maps are linear on cones of the mutation fan, so there are finitely many distinct vectors~$\eta_\kk^B(\a)$ as $\kk$ varies over all sequences if the same statement is true for the shortest integer vectors spanning the rays of $C$.

Mutation fans depend only on the exchange matrix and $B$ is of finite mutation type. 
Therefore the set of imaginary walls $\d_\infty^{B'}$ as $B'$ varies over all matrices $B'$ mutationally equivalent to $B$ is finite and mutation maps permute this set.
Further, each imaginary wall contains finitely many rays and mutation maps permute the set of shortest integer vectors in rays of the imaginary walls.
We conclude that each shortest integer vector in a ray of $C$ has finitely many images under mutation maps. 
\end{proof}

\begin{remark}\label{correction}
As already mentioned, Theorem~\ref{mutfin coeff} agrees with the results of \cite{FeTuCoeff} that are stated in terms of affine type.
However, Theorem~\ref{mutfin coeff} constitutes a minor correction to the affine cases of \cite[Theorem~3.2]{FeTuCoeff} and \cite[Theorem~9.2]{FeTuCoeff}, which are stated in terms of marked surfaces and orbifolds.
\newword{Peripheral laminations} are defined in~\cite{FeTuCoeff} to be laminations such that every curve can be isotopically deformed to all or part of some boundary component of the surface, and \cite[Theorem~3.2]{FeTuCoeff} states that a vector is allowable if and only if it is the shear coordinate vector of a peripheral lamination.
However, in the twice-punctured disk (corresponding to affine type $\widetilde D$), Theorem~\ref{mutfin coeff} says that a vector is allowable if and only if it is the shear coordinate vector of a lamination such that every curve either can be isotopically deformed to all or part of the boundary or has endpoints that spiral into the two punctures.
As a fix to \cite[Theorem~3.2]{FeTuCoeff}, we propose to redefine peripheral laminations as laminations $\Lambda$ such that, for every closed curve $C$ that can appear in a lamination, the union $\Lambda\cup\set{C}$ is a lamination (or in other words, $C$ is compatible with every curve in $\Lambda$).
One can check that the twice-punctured annulus is the unique marked surface where these two definitions differ.  
The proof of \cite[Theorem~3.2]{FeTuCoeff} goes through as written, with this alternative definition of peripheral laminations.
The same change to the definition works for marked orbifolds in \cite[Theorem~9.2]{FeTuCoeff}, and the two definitions coincide for orbifolds except in the cases of a disk with 2 ``special points'' (punctures or orbifold points), which are all of affine type.
\end{remark}

\subsection*{Acknowledgments}
We thank Anna Felikson and Pavel Tumarkin for helpful comments on a draft of this paper.

\bibliographystyle{plain}
\bibliography{bibliography}
\vspace{-0.175 em}

\end{document}